\documentclass{tsp-short}
\newtheorem{thm}{Theorem}[section]
\newtheorem{lem}{Lemma}[section]

\theoremstyle{remark}
\newtheorem{rem}{Remark}[section]
\theoremstyle{definition}

\RequirePackage[colorlinks,citecolor=blue,linkcolor=blue,urlcolor=blue]{hyperref}

\makeatletter
\def\namedlabel#1#2{\begingroup
    #2%
    \def\@currentlabel{#2}%
    \phantomsection\label{#1}\endgroup
}
\makeatother

\DeclareMathOperator{\pr}{pr}
\DeclareMathOperator{\leb}{Leb}

\newcommand{\Li}{L_2^{\uparrow}}
\newcommand{\Q}{\mathbb{Q}}
\newcommand{\R}{\mathbb{R}}
\newcommand{\I}{\mathbb{I}}

\newcommand{\N}{\mathbb{N}}
\newcommand{\F}{\mathcal{F}}
\newcommand{\p}{\mathbb{P}}

\newcommand{\Df}{\mathcal{D}}
\newcommand{\mC}{\mathcal{C}}

\newcommand{\RR}{\mathcal{R}}
\newcommand{\E}{\mathbb{E}}
\newcommand{\eps}{\varepsilon}
\newcommand{\cdl}{c\`{a}dl\`{a}g }

\begin{document}

\title
[On number of particles in CFWD]
{On Number of Particles in \\ Coalescing-Fragmentating Wasserstein Dynamics}

\author{Vitalii Konarovskyi}
\address{Faculty of Mathematics, Computer Science and Natural Sciences, University of Hamburg, Bundesstraße 55, 20146 Hamburg, Germany; Institute of Mathematics, University of Leipzig, Augustusplatz 10, 04109 Leipzig, Germany; Institute of Mathematics of NAS of Ukraine, Tereschenkivska st. 3, 01024 Kiev, Ukraine}
\curraddr{Faculty of Mathematics, Computer Science and Natural Sciences, University of Hamburg, Bundesstraße 55, 20146 Hamburg, Germany}
\email{konarovskyi@gmail.com}
\thanks{The work is supported by the Grant ``Leading and Young Scientists Research Support'' No.~2020.02/0303. 
The author thanks the referee for the careful reading of the paper and many valuable comments.}


\begin{abstract}
We consider the system of sticky-reflected Brownian particles on the real line proposed in~\cite{Konarovskyi:CFWDPA:2017}. The model is a modification of the Howitt-Warren flow but now the diffusion rate of particles is inversely proportional to the mass which they transfer.  It is known that the system consists of a finite number of distinct particles for almost all times. In this paper, we show that the system also admits an infinite number of distinct particles on a dense subset of the time interval if and only if the function responsible for the splitting of particles takes an infinite number of values. 
\end{abstract}

\maketitle

\section{Introduction}%
\label{sec:introduction}

In the paper we study the interacting particle system on the real line which intuitively can be described as follows. Diffusion particles start at some finite or infinite family of points and move independently until their meeting. Every particle transfer a mass and its diffusion rate is inversely proportional to its mass. When particles meet they sticky-reflect from each other. The evolution of the particle system is similar to the motion of particles in the Howitt-Warren flow~\cite{Howitt:2009}. The main difference is that the motion of particles in our system inversely-proportionally depends on their mass. In particular, particles with ``infinitesimally small'' mass have ``infinitely large'' diffusion rate. We call this model the {\it coalescing-fragmentating Wasserstein dynamics} (CFWD).

More precisely, let $X(u,t)$ be a position of particle labeled by $u \in (0,1)$ (we will shortly say ``particle $u$'') at time $t\geq 0$, and $m(u,t)$ be its mass that is the Lebesgue measure $\leb$ of the corresponding cluster $\pi(u,t)=\{ v \in (0,1):\ X(u,t)=X(v,t) \}$. Assume that the diffusion rate of the particle $u$ at time $t$ is inversely proportional to its mass $m(u,t)$. The sticky-reflecting interaction between particles is defined by the drift 
\[
  \xi(u)- \frac{1}{ m(u,t) }\int_{ \pi(u,t) }  \xi(v)dv
\]
with a fixed bounded non-decreasing right-continuous function $\xi$ called the interaction potential. Indeed, if $\xi$ is constant on $\pi(u,t)$ or $\pi(u,t)$ is a one point set then the particle $u$ has no drift. Otherwise, particles which stay together will have different drift for corresponding different values of $\xi$ on $\pi(u,t)$ that makes particles to split. We remark that the order between particles is preserved. Therefore, we may assume that $X(u,t)\leq X(v,t)$ for all $u<v$ and $t\geq 0$.

In~\cite{Konarovskyi:CFWDPA:2017}, we showed that $X$ appears as a martingale solution to the following SDE 
\begin{equation}\label{eq:mainequation}
\begin{split}
  dX_t&=\pr_{X_t}dW_t+(\xi-\pr_{X_t}\xi)dt,\quad t\geq 0,\\
X_0&=g,
\end{split}
\end{equation}
in the space $\Li$ of all square-integrable functions (classes of equivalences) $f:(0,1)\to \R$ which have a non-decreasing version, $W_t$, $t\geq 0$, is a cylindrical Wiener process in $L_2=L_2([0,1],du)$, $X_t:=X(\cdot ,t) \in \Li$, and $\pr_f$ denotes the orthogonal projection in $L_2$ onto its subspace $L_2(f)$ of all $\sigma(f)$-measurable functions. The function $g \in \Li$ describes the initial position of particles. 

  \begin{rem} 
  \label{rem_right_continuous_version}
  For convenience of notation, considering $f \in \Li$ as a function, we will always take its right continuous version on $(0,1)$, which is unique according to, e.g., Proposition~A.1~\cite{Konarovskyi:EJP:2017} and Remark~A.6 ibid. 
  \end{rem}

The existence result in~\cite{Konarovskyi:CFWDPA:2017} claims that for every $g \in \Li$ satisfying $\int_{ 0 }^{ 1 } g^{2+\eps}(u)du<\infty $ for some $\eps>0$ there exist an $L_2$-valued cylindrical Wiener process $W_t$, $t\geq 0$, and a continuous $\Li$-valued process $X_t$, $t\geq 0$, both defined on the same filtered probability space $(\Omega,\F,(\F)_{t\geq 0},\p)$ such that $\E \|X_t\|_{L_2}^2<\infty$, $t\geq 0$, and 
\[
  X_t=g+\int_{ 0 }^{ t } \pr_{X_s}dW_s+\int_{ 0 }^{ t }(\xi-\pr_{X_s}\xi)ds, \quad t\geq 0.
\]
We call such $X_t$, $t\geq 0$, a {\it weak solution} to~\eqref{eq:mainequation} and assume that the process in a mathematical description of the CFWD.

Since $X_t(\omega)$ is a class of equivalences, the value $X_t(u,\omega)$ is not well defined for every $u \in (0,1)$. Note that $X_t(\omega)$ is an element of $\Li$, therefore, one can easily see that there exists its unique right-continuous modification, denoted also by $X_t(\omega)$. In agreement with Remark~\ref{rem_right_continuous_version}, we will hereinafter consider this right-continuous  version of $X$. In general, the process $X_t(u)$, $t\geq 0$, is not continuous. But it turns out that under some regularity conditions of the initial condition $g$ and the interaction potential $\xi$ one can show that $X(u,\cdot )=X_t(u)$ is a continuous semi-martingale which satisfies some natural conditions for each $u$. We will not use those conditions for the proof of our main result stated in Theorem~\ref{thm_main_theorem}, however, we will provide them here to help the reader better understand the particle model. 

Let $\Df\big([0,1],\mC[0,\infty)\big)$ denote the Skorohod space of all \cdl functions from $[0,1]$ to the space $\mC[0,\infty)$ of real-valued continuous functions defined on $[0,\infty)$.  If the initial condition $g$ and the interaction potential $\xi$ are right-continuous and piecewise $(\frac{1}{ 2 }+)$-H\"older continuous\footnote{There exist $\eps>0$ and a finite partition of the interval $[0,1]$ such that the functions are $\left(\frac{1}{ 2 }+\eps\right)$-H\"older continuous on each interval of the partition}, then equation~\eqref{eq:mainequation} admits a weak solution with a modification $\{X(u,t),\ t\geq 0,\ u \in [0,1]\}$ from $\Df\big([0,1],\mC[0,\infty)\big)$ satisfying the following properties
  \begin{enumerate}
    \item [(\namedlabel{R1}{R1})] for all $u \in [0,1]$, $X(u,0)=g(u)$;

    \item [(\namedlabel{R2}{R2})] for each $u<v$ from $[0,1]$ and $t\geq 0$, $X(u,t)\leq X(u,t)$;

    \item [(\namedlabel{R3}{R3})] the process 
      \[
	M(u,t):=X(u,t)-g(u)-\int_{ 0 }^{ t } \left( \xi(u)- \frac{1}{ m(u,s) }\int_{ \pi(u,s) }\xi(v)dv  \right)ds, \quad t\geq 0,
      \]
      is a continuous square integrable martingale with respect to the filtration $\F_t=\sigma(X(v,s),\ v \in [0,1],\ s\leq t)$, $t\geq 0$;

    \item [(\namedlabel{R4}{R4})] the joint quadratic variation of $M(u,\cdot )$ and $M(v,\cdot )$ equals 
      \[
	\langle M(u,\cdot ),M(v,\cdot)\rangle_t= \int_{ 0 }^{ t } \frac{ \I_{\left\{ X(u,s)=X(v,s) \right\}} }{ m(u,s) }ds,\quad t\geq 0.
      \]
  \end{enumerate}

  We remark that the uniqueness of a weak solution to equation~\eqref{eq:mainequation} remains an important open problem. For interested readers we would like to pointed out that the CFWD admits an invariant measure and its reversible version was stydied in~\cite{Konarovskyi:CFWD:2017}. Its connection with the Wasserstein diffusion~\cite{Renesse:2009} and the geometry of the Wasserstein space of probability measures on the real line also were studied there.

  We will denote by $\sharp f$ a number of distinct values of $f \in \Li$, which is well-defined according to Remark~\ref{rem_right_continuous_version}. By Lemma~6.1~\cite{Konarovskyi:EJP:2017}, the square of the Hilbert-Schmidt norm of the orthogonal projection $\pr_{f}$ coinsides with $\sharp f$, i.e.
  \begin{equation} 
  \label{equ_hilber_schimdt_norm}
  \|\pr_f\|_{HS}^2:=\sum_{ n=1 }^{ \infty } \|\pr_{f}e_n\|_{L_2}^2=\sharp f,
  \end{equation}
  where $\{ e_n,\ n\geq 1 \}$ is an orthonormal basis in $L_2$.  Therefore, we can interpret the random variable $\|\pr_{X_t}\|^2_{HS}=\sharp X_t$ as a number of distinct particles in CFWD at time $t\geq 0$. In particular, if $X_t=X(\cdot,t)$, $t\geq 0$, where the random element $\{ X(u,t),\ t\geq 0,\ u \in [0,1] \}$ in $\Df\big([0,1],\mC[0,\infty)\big)$ satisfies conditions~\eqref{R1}-\eqref{R4}, then $\sharp X(\cdot,t)$ is exactly the number of distinct particles at time $t\geq 0$ in the CFWD. Since $X_t$, $t\geq 0$, is square integrable and $\xi$ is bounded, Theorem~2.4~\cite{Gawarecki:2011} and equality~\eqref{equ_hilber_schimdt_norm} imply 
  \[
    \int_{ 0 }^{ t }\E(\sharp X_s) ds<\infty  
  \]
  for all $t\geq 0$. This yields that 
  \begin{equation} 
  \label{equ_finiteness_of_sharp_x}
  \p\left\{ \sharp X_t <\infty\ \mbox{for a.e.}\ t \in [0,\infty) \right\}=1,
  \end{equation}
  i.e. the CFWD consists of a finite number of particles at almost all times with probability~1. The goal of this paper is to show that almost surely there exists a (random) dense subset of the time interval $[0,\infty)$ on which the CFWD has an infinite number of particles if and only if $\sharp\xi=\infty$. We remark that the property $\sharp\xi=\infty$ is equivalent to the fact that $L_2(\xi)$ is infinite dimensional, by~\eqref{equ_hilber_schimdt_norm}.

  \begin{thm} 
  \label{thm_main_theorem}
    \begin{enumerate}
      \item [(i)] If $\sharp\xi=+\infty$, then almost surely there exists a (random) dense subset $S$ of $[0,\infty)$ such that $\sharp X_t=\infty$, $t \in S$, that is,
	\[
	  \p\left\{ \exists S\ \mbox{dense in}\ [0,\infty)\ \mbox{such that}\ \sharp X_t=\infty,\ t \in S \right\}=1.
	\]

      \item [(ii)] If $\sharp\xi<\infty$, then
	\begin{equation} 
  \label{equ_finiteness_of_number_of_particles}
	  \p\left\{ \sharp X_t<\infty, \ \ t \in [0,\infty) \right\}=1.
	\end{equation}
    \end{enumerate}
  \end{thm}
  We note that the CFWD coincides with the modified massive Arratia flow~\cite{Konarovskyi:EJP:2017,Konarovskyi:AP:2017,Konarovskyi:CD:2020,Konarovskyi:CPAM:2019,Marx:2018} for $\xi=0$. In this case, equality~\eqref{equ_finiteness_of_number_of_particles} was stated in~\cite[Proposition~6.2]{Konarovskyi:EJP:2017}.

\section{Auxiliary statements}

Let $\mC\big([a,b],\Li\big)$ denote the space of continuous functions from $[a,b]$ to $\Li$ endowed with the usual topology. We recall that the map $h \mapsto \|\pr_hf\|_{L_2}$ from $\Li$ to $\R$ is lower semi-continuous for each $f \in L_2$, that is, 
\begin{equation} 
  \label{equ_lover_semicontinuity}
  \|\pr_hf\|_{L_2}\leq\varliminf_{n\to \infty}\|\pr_{h_n}f\|_{L_2},\quad \mbox{as}\ \ h_n \to h\ \ \mbox{in}\ \ \Li.
\end{equation}
The proof of this fact can be found in~\cite[Lemma~A.4]{Konarovskyi:CFWDPA:2017}. By Fatou's lemma, the map $h \mapsto \|\pr_{h}\|_{HS}$ is lower semi-continuous as well.

The following lemma is needed for the measurability of events which will appear in the proof of Theorem~\ref{thm_main_theorem}.

\begin{lem}\label{lem:measurability}
For each $[a,b]$, the map $f\mapsto \sup_{t\in [a,b]}\|f_t\|_{HS}$ from $\mC\big([a,b],\Li\big)$ to $\R\cup \{+\infty\}$ is measurable. 
\end{lem}

\begin{proof}
  Let $t \geq 0$ be fixed. Note that the map $f\mapsto\|\pr_{f_t}\|_{HS}$ from $\mC\big([a,b],\Li\big)$ to $\R$ is lower semi-continuous because it is the composition of the continuous map $\mC\big([a,b],\Li\big)\ni g\mapsto g_t \in \Li$ and the lower semi-continuouous map $\Li \ni h\mapsto \|\pr_{h}\|_{HS} \in \R$. This yields the claim of the lemmas due to the measurability of $f\mapsto\|\pr_{f_t}\|_{HS}$ and the equality 
\begin{align*}
  \{f:\ \sup_{t\in [a,b]}\|f_t\|_{HS}\leq c\}&=\bigcap_{t\in[a,b]\cap\Q}\left\{f:\ \|\pr_{f_t}\|_{HS}\leq c\right\},
\end{align*}
for all $c\geq 0$.
\end{proof}

The following lemma directly follows from the lower semi-continuity of the map $t\mapsto \|\pr_{f_t}\|_{HS}$ for every $f \in \mC\big([0,\infty),\Li\big)$.

\begin{lem}\label{lem:closabilityofA}
  For every $f\in \mC\big([0,\infty),\Li\big)$, $c\geq 0$ and $0 \leq a<b$ the set $A_c^{f,a,b}:=\big\{t\in[a,b]:\ \|\pr_{f_t}\|^2_{HS}\leq c\big\}$ is closed in $[0,\infty)$. 
\end{lem}

We will also need a property of a function $f \in \mC\big([0,\infty),\Li\big)$ if the Hilbert-Schmidt norm $\|\pr_{f_t}\|_{HS}$, $t \in [0,\infty)$, is constant on an interval.

\begin{lem}\label{lem:coalescing}
  Assume that $f$ belongs to $\mC\big([0,\infty),\Li\big)$ and $\|\pr_{f_t}\|_{HS}$, $t\in[a,b]$, is constant for some $0\leq a<b$. Then 
  \begin{enumerate}
    \item [(i)] for every $u_0 \in (0,1)$ there exist $u_1<u_0<u_2$ and $\alpha<\beta$ from $[a,b]$ such that $f_t$ is constant on $[u_1,u_0)$ and $[u_0,u_2)$ for each $t\in[\alpha,\beta]$;

    \item [(ii)] for $u_0=0$ (resp. $u_0$=1) there exist $u_2>u_0$ (resp. $u_1<u_0$) and $\alpha<\beta$ from $[a,b]$ such that $f_t$ is constant on $[u_0,u_2)$ (resp. on $[u_1,u_0]$) for each $t\in[\alpha,\beta]$.
  \end{enumerate}
\end{lem}

\begin{proof}
  Since $\|\pr_{f_{\cdot}}\|_{HS}$ is constant on $[a,b]$, the function $f_t$ takes a fixed number of distinct values, denoted by $n$, for each $t\in[a,b]$, by equality~\eqref{equ_hilber_schimdt_norm}. Let
\[
f_t=\sum_{k=1}^nx_k(t)\I_{[q_{k-1}(t),q_k(t))},\quad t\in[a,b],
\] 
where $x_1(t)<\ldots<x_n(t)$ and $0=q_0(t)<q_1(t)<\ldots<q_n(t)=1$.  From continuity of $f_t$, $t\geq 0$, it follows that the functions $x_k$ and $q_k$ are continuous on $[a,b]$ for each $k$ in $[n]:=\{1,\dots,n\}$.

If there exists $l \in [n]$ such that 
\begin{equation} 
  \label{equ_interval}
  u_0 \in (q_{l-1}(t),q_{l}(t))\quad \mbox{for some}\ \ t \in (a,b),
\end{equation}
then one can take $u_1<u_0<u_2$ and $\alpha<\beta$ from $[a,b]$ satisfying $u_1,u_2$ in $(q_{l-1}(t),q_l(t))$ for all $t \in [\alpha,\beta]$, by the continuity of $q_k$, $k \in [n]$. This trivially implies the statement of the lemma. If $l$ satisfying~\eqref{equ_interval} does not exist, then $u_0=q_l(t)$ for some $l \in [n]\cup \{0\}$ and all $t \in [a,b]$, which also yields the statement.
\end{proof}

\section{Proof of Theorem~\ref{thm_main_theorem}}

We first consider the case $\sharp \xi=\infty$. In order to show that with probability 1 there exists a dense subset $S$ of $[0,\infty)$ such that $\sharp X_t=\infty$ for all $t \in S$, it is enough to prove that for each $0\leq a<b$ one has
\begin{equation} 
  \label{equ_infinite_number_in_proof}
  \p\left\{\sup_{t\in[a,b]}\sharp X_t=\infty\right\}=1.
\end{equation}
Recall that the measurability of $\sup_{t\in[a,b]}\sharp X_t$ follows from Lemma~\ref{lem:measurability} and equality~\eqref{equ_hilber_schimdt_norm}. We suppose that equality~\eqref{equ_infinite_number_in_proof} is false, that is,
$$
\p\left\{\sup_{t\in[a,b]}\sharp X_t<\infty\right\}>0.
$$
Setting $A_n^{a,b}(\omega):=\left\{t\in[a,b]:\ \|\pr_{X_t(\omega)}\|^2_{HS}\leq n\right\}$, $\omega \in \Omega$, and using equality~\eqref{equ_hilber_schimdt_norm}, we can conclude that
$$
\p\left\{\bigcup_{n=1}^{\infty}A^{a,b}_n=[a,b]\right\}>0.
$$ 
By Lemma~\ref{lem:closabilityofA} and the Baire category theorem, we have
$$
\p\big\{\exists a_1<b_1\ \ \mbox{from}\ \ [a,b]\ \  \mbox{and}\ \ n\in\N\ \  \mbox{such that}\ \ \|\pr_{X_t}\|_{HS}^2\leq n,\ \  t\in[a_1,b_1]\big\}>0.
$$
Consequently, we can find non-random $a_1<b_1$ from $[a,b]$ and $k_1\in\N$ such that 
$$
\p\Big\{\|\pr_{X_t}\|_{HS}^2\leq k_1,\ \ t \in [a_1,b_1]\Big\}>0.
$$
Let $k_2 \in [k_1]$ be the minimal number such that
$$
\p\Big\{\|\pr_{X_t}\|_{HS}^2\leq k_2,\ \ t\in[a_1,b_1]\Big\}>0.
$$
By the minimality of $k_2$, we can conclude that
$$
\p\left\{A^{a_1,b_1}_{k_2}\setminus A^{a_1,b_1}_{k_2-1}\not=\emptyset\right\}>0,
$$
where $A^{a_1,b_1}_0=\emptyset$ if $k_2=1$.  Next, since $A^{a_1,b_1}_{k_2}\setminus A^{a_1,b_1}_{k_2-1}$ is open in $A^{a_1,b_1}_{k_2}=[a_1,b_1]$ and non-empty with positive probability, one can find non-random $a_2<b_2$ from $[a_1,b_1]$ satisfying
$$
\p\Big\{\|\pr_{X_t}\|_{HS}^2=k_2,\ \ t\in[a_2,b_2]\Big\}>0.
$$

Next, due to the equality $\sharp\xi=\infty$, there exists $u_0 \in [0,1]$ such that $\xi$ takes an infinite number of distinct values in $[u_1,u_0)$ for all $u_1<u_0$ or in $[u_0,u_2)$ for all $u_2>u_0$. Using Lemma~\ref{lem:coalescing} and the monotonicity of $X_t(\omega)$ for all $t$ and $\omega$, one can find non-random $a_3<b_3$ from $[a_2,b_2]$ and $u<v$ such that $u=u_0$ or $v=u_0$, $\xi$ takes an infinite number of distinct values on $[u,v]$ and
$$
\p\big\{X_t(u)=X_t(\tilde{u}),\  \ \tilde{u} \in [u,v),\ \ t\in[a_3,b_3]\big\}>0.
$$

Let $h:=\I_{[(u+v)/2,v)}-\I_{[u,(u+v)/2)}$. Since $X_t$, $t\geq 0$, solves equation~\eqref{eq:mainequation} and belongs to $\Li$, one has that  $( X_t,h)_{L_2}$, $t\geq 0$, is a continuous non-negative process such that  
$$
M_h(t)=( X_t,h)_{L_2}-\int_0^t\left(\xi-\pr_{X_s}\xi,h\right)_{L_2} ds,\quad t\geq 0,
$$
is a continuous square integrable $(\F_t)$-martingale with quadratic variation 
\[
  \langle M_h\rangle_t=\int_0^t\|\pr_{X_s}h\|_2^2ds,\quad t\geq 0.
\]
We take $\omega$ from the event 
\[
  A:=\left\{\forall t\in[a_3,b_3]\ \  X_t\ \ \mbox{is constant on}\ \ [u,v)\right\},
\]
and note that $( X_t(\omega),h)_{L_2}=0$, $\pr_{X_s(\omega)}h=0$ and
$$
(\pr_{X_s}\xi,h)_{L_2}=(\xi,\pr_{X_s}h)_{L_2}=0
$$ 
for all $s\in[a_3,b_3]$ due to the choice of $h$. Thus, we can conclude that
\begin{equation} 
  \label{equ_mh}
  M_h(t,\omega)=-\int_0^{a_3}(\xi-\pr_{X_s(\omega)}\xi,h)_{L_2} ds-\int_{a_3}^t(\xi,h)_{L_2} ds
\end{equation}
and
$$
\langle M_h\rangle_t(\omega)=\int_0^{a_3}\|\pr_{X_s(\omega)}h\|_2^2ds
$$
for all $t\in[a_3,b_3]$. The equality for the quadratic variation of $M_h$ and the representation of continuous martingales as a time changed Brownian motion (see~\cite[Theorem~II.7.2']{Ikeda:1989}) imply that $M_h(t,\omega)=M_h(a_3,\omega)$, $t\in[a_3,b_3]$ for a.e. $\omega \in A$. Since the non-decreasing function $\xi$ is not a constant on $[u,v]$, the inner product $( \xi,h )_{L_2}$ is strictly positive due to the choice of $h$. According to equality~\eqref{equ_mh}, $M_h(t,\omega)$, $t\in[a_3,b_3]$, is strictly increasing (in $t$) for a.e. $\omega \in A$ because $(\xi,h)_{L_2}>0$. Since $\p\{A\}>0$, we get a contradiction. This completes the proof of the first part of the theorem.

We next prove claim (ii). Due to $\sharp \xi<\infty$, there exists a finite partition $\pi_k$, $k \in [n]$, of the interval $[0,1)$ by intervals of the form $[a,b)$ such that 
\[
  \xi(u)=\sum_{ k=1 }^{ n } \xi_k\I_{\pi_k}(u), \quad u \in [0,1).
\]
In order to prove (ii), it is enough to show that almost surely $X_t$ takes a finite number of distinct values on every interval $\pi_k$. We fix $k \in [n]$ and consider the countable family of functions $h_{u,v}:=\I_{[(u+v)/2,v)}-\I_{[u,(u+v)/2)}$ from $L_2$, $u,v \in \pi_k \cap \Q$, denoted by $\RR$. 

We first remark that for every $h \in \RR$ the process $( X_t,h )_{L_2}$, $t\geq 0$, is a non-negative continuous supermartingale. Indeed, the non-negativity  follows from the inequality $(f,h)_{L_2}\geq 0$ for every $f \in \Li$ and $h \in \RR$. In order to show that $(X_t,h)_{L_2}$, $t\geq 0$, is a supermartingale, we use the fact that it is a weak martingale solution to equation~\eqref{eq:mainequation}. Hence for each $h \in \RR$
\begin{align*}
  ( X_t,h )_{L_2}=M_{h}(t)+ \int_{ 0 }^{ t } \left(\xi-\pr_{X_s}\xi,h\right)_{L_2}ds= M_h(t)-\int_{ 0 }^{ t }  \left(\pr_{X_s}\xi,h\right)_{L_2}ds, \quad t\geq 0,
\end{align*}
where $M_h$ is a martingale. According to Lemma~A.2~\cite{Konarovskyi:CFWDPA:2017}, the orthogonal projection $\pr_f$ maps the space $\Li$ into $\Li$ for every $f \in \Li$. Hence, $\pr_{X_s}\xi \in \Li$ and, therefore, $\left( \pr_{X_s}\xi,h \right)\geq 0$. This implies that $(X_t,h)_{L_2}$, $t\geq 0$, is a continuous supermartingale. 

It is well known that hitting at zero a positive continuous suparmartingales stays there forewer (see e.g. Proposition~II.3.4~\cite{Revuz:1999}). We denote the corresponding event for the supermartingale $( X_t,h )_{L_2}$, $t\geq 0$, by $\Omega_h$, i.e
\[
  \Omega_h=\left\{ 
  \begin{array}{l}
  	\mbox{for every} \ t \in [0,\infty)\ \mbox{the equality}\ ( X_t,h )_{L_2}=0 \\
  	 \mbox{implies} \ (X_s,h)_{L_2}=0 \ \mbox{for all} \ s\geq t \\
  \end{array}
\right\}.
\]
Then $\p\{ \Omega_h \}=1$ for every $h \in \RR$. Thus, the event $\Omega':=\bigcap_{ h \in \RR } \Omega_h$ has the probability~1. Take $\omega \in \Omega'$, $u,v \in (0,1)$, and $t\geq 0$ such that $X_t(u,\omega)=X_t(v,\omega)$. Then for every $h \in \RR$ one has $( X_t(\omega),h )_{L_2}=0$ and, consequently, $( X_s(\omega),h )_{L_2}=0$ for all $s\geq t$, by the choice of $\omega$. Using the right continuity of $X_s(\cdot,\omega)$ (see Remark~\ref{rem_right_continuous_version}), it is easily seen that $X_s(u)=X_s(v)$, $s\geq t$. In other words, the process $X_t$, $t\geq 0$, satisfies the following property: if $X_t$ is constant on an interval $[u,v] \in \pi_k$ for some $k \in [n]$, then it remains constant on this interval for every $s\geq t$.  Combining this coalescing property of $X_t$, $t\geq 0$, on every interval $\pi_k$, $k \in [n]$ with equality~\eqref{equ_finiteness_of_sharp_x}, we get claim (ii) of the theorem.


\providecommand{\bysame}{\leavevmode\hbox to3em{\hrulefill}\thinspace}
\providecommand{\MR}{\relax\ifhmode\unskip\space\fi MR }
\providecommand{\MRhref}[2]{%
  \href{http://www.ams.org/mathscinet-getitem?mr=#1}{#2}
}
\providecommand{\href}[2]{#2}

\end{document}